\documentclass[12pt]{amsart}
\usepackage{amssymb}
\usepackage[all]{xy}
\newtheorem{theorem}{Theorem}[section]
\newtheorem{lemma}[theorem]{Lemma}
\newtheorem{proposition}[theorem]{Proposition}
\newtheorem{definition}[theorem]{Definition}
\newtheorem{remark}[theorem]{Remark}

\newtheorem{example}[theorem]{Example}
\numberwithin{equation}{section}
\def\to{\rightarrow}
\def\f{\mathfrak}
\def\m{\mathbb}
\def\c{\mathcal}
\def\b{\mathbf}
\def\r{\mathrm}
\def\ot{\otimes}
\def\i{\mathrm{i}}
\begin{document}
\baselineskip15pt
\title[Quantum groups of maps]{On the quantum groups and semigroups of maps between noncommutative spaces}
\author[M. M. Sadr]{Maysam Maysami Sadr}
\address{Department of Mathematics\\
Institute for Advanced Studies in Basic Sciences \\
P.O. Box 45195-1159, Zanjan 45137-66731, Iran}
\email{sadr@iasbs.ac.ir}
\begin{abstract}
We define algebraic families of (all) morphisms which are purely algebraic analogs of
quantum families of (all) maps introduced by P.M. So{\l}tan. Also, algebraic families of (all) isomorphisms are introduced.
By using these notions we construct two classes of Hopf-algebras which may be interpreted as the quantum group
of all maps from a finite space to a quantum group, and the quantum group of all automorphisms of a finite NC space.
As special cases three classes of NC objects are introduced:
quantum group of gauge transformations, Pontryagin dual of a quantum group, and Galois-Hopf-algebra of an algebra extension.
\end{abstract}
\maketitle
\section{Introduction}
Our work is based on the extension to the noncommutative setting of the following picture.
Let $X$ and $Y$ be compact Hausdorff spaces. If $S$ is a Hausdorff space then there is a canonical one-to-one correspondence
between continuous maps $f:S\times X\to Y$ and continuous families $\bar{f}=\{f(s,?)\}_{s\in S}$ of continuous maps from $X$ to $Y$
with parameter space $S$. Moreover, by Exponential Law of Topology, there is a canonical homeomorphism $Y^{S\times X}\simeq (Y^X)^S$
where the mapping spaces are endowed with compact-open topology.
Indeed, the family of all continuous maps from $X$ to $Y$, that is the family $\bar{e}$ induced by the evaluation map
$e:Y^X\times X\to Y$ defined by $e(a,x)=a(x)$, and the space $Y^X$ are completely characterized by the following universal property.
\begin{quote}
For every Hausdorff space $S$ and any continuous map $f:S\times X\to Y$ there is a unique continuous map
$\bar{f}$ that makes the following diagram commutative.
$$\xymatrix{Y^X\times X\ar[r]^-{e}& Y\\
S\times X\ar[u]^-{\bar{f}\times\r{id}_X}\ar[-1,1]_{f}&}$$
\end{quote}
In the context of NC Geometry and Operator Algebra
P.M. So{\l}tan \cite{Soltan1} and S.L. Woronowicz \cite{Woronowicz1} have introduced the notion of
\emph{quantum space of all maps} between two C*-algebraic NC spaces.
To define this notion they used the C*-dual of the above universal property: For two quantum spaces $\f{Q}A$ and $\f{Q}B$,
a quantum space $\f{Q}C$ together with a *-homomorphism $\phi:B\to A\ot C$ is called the quantum family of all maps from $\f{Q}A$ to $\f{Q}B$
if for every *-homomorphism $\psi:B\to A\ot D$, between C*-algebras, there is a unique *-homomorphism $\bar{\psi}:C\to D$ satisfying
$(\r{id}_A\ot\bar{\psi})\phi=\psi$. So{\l}tan \cite{Soltan1} has showed that such universal C*-algebra $C$ and *-homomorphism $\phi$ exist when
$\f{Q}A$ is a finite quantum space i.e. $A$ is a finite dimensional C*-algebra. He also showed (\cite{Soltan1}, \cite{Soltan2})
that if $B=A$ then $C$ has a canonical C*-bialgebra
structure (as in the classical case the space $X^X$ is a topological semigroup) and the bialgebra structure of some classes of quantum groups,
e.g. quantum permutation groups (\cite{Wang1}, \cite{BanicaBichonCollins1}) and quantum isometry groups (\cite{Banica1}), are the same as of $C$.

The aim of this paper is to study the purely algebraic version of the quantum space of all maps and its bialgebra structures.
It seems that our algebraic formalism is more interesting and useful than C*-formalism because the class of
finite dimensional C*-algebras restricts only to finite direct sums of full matrix algebras over the complex field.

In Section 2 we define \emph{algebraic family of all morphisms} from an algebra $B$ to another algebra $A$
over a commutative ring $\m{K}$. This will be a morphism $\f{m}(B,A):B\to A\ot\f{M}(B,A)$ between algebras which satisfies a universal property
analogous to the above. As we mentioned above, this is a purely algebraic reformulation of the notion of quantum family of all maps
\cite{Soltan1}. Analogue of Theorem 3.3 of \cite{Soltan1}, we show that $\f{m}(B,A)$ exists if $A$ is a free finite
rank $\m{K}$-module. Also we consider some elementary properties of $\f{M}(?,A)$ as a functor on the category of algebras.

In Section 3 we define \emph{algebraic families of all isomorphisms}. The idea of the definition follows from the following
trivial fact: A family $\bar{f}$ associated to the continuous map $f:S\times X\to Y$ is a family of homeomorphisms from $X$ onto $Y$ if and only if
there is another family $\bar{g}$ associated to a continuous map $g:S\times Y\to X$ such that $g(s,f(s,x))=x$ and $f(s,g(s,y))=y$ for every
$x\in X$, $y\in Y$ and $s\in S$. It follows that the family of all homeomorphisms must be universal with respect to this property. Then by dualizing
this universal property we find the desired notion of the algebraic family of all isomorphisms $\f{i}(B,A):B\to A\ot\f{I}(B,A)$ from $B$ onto $A$.
We will show that $\f{i}(B,A)$ exists if both $A$ and $B$ are free finite rank $\m{K}$-modules.

In Section 4 we consider existence of dual of the following classical objects:
\begin{enumerate}
\item[(1)] The space of all bundle morphisms between two fibre bundles over a classical space.
\item[(2)] The space of all continuous maps between two classical spaces which are identity over a common subspace.
\item[(3)] The space of all continuous group homomorphisms between two topological groups.
\item[(4)] The space of all $G$-maps between two $G$-spaces for a topological group $G$.
\end{enumerate}

If $G$ is a compact group then $G^X$ is a topological group for every compact space $X$. Section 5 is devoted to the noncommutative
analog of this fact. More precisely, we will show that if $B$ is a Hopf-algebra and $A$ is commutative and finite-rank then $\f{M}(B,A)$
has a canonical Hopf-algebra structure. In the C*-case this construction has been considered in \cite{Soltan3} and \cite{Sadr1}.

In Sections 6 and 7, by using the results of the preceding sections, we construct two classes of Hopf-algebras which may be interpreted
as the quantum group of gauge transformations over trivial quantum vector bundles in the sense of \cite{BrzezinskiMajid1} and, as the
Pontryagin dual of finite-rank Hopf-algebras.

If $X$ is a compact Hausdorff space then the space of all homeomorphisms of $X$ is a topological group. In Section 8 we show
this is also the case in the dual context. Indeed, we will show that for any finite-rank Hopf-algebra $A$ the algebra $\f{I}(A,A)$ has
a canonical Hopf-algebra structure. Finally, we consider a notion dual to the notion of Galois group of algebra extensions.

We announce that there is a C*-algebraic concept of \emph{family of quantum invertible maps} due to
Podle\'{s} \cite{Podles1}, \cite{SkalskiSoltan1} which is different from our notion of algebraic family of isomorphisms
and because of its topological nature it seems that its purely algebraic formulation is not interesting. Also, after preparation of this paper
we received a recent paper \cite{SkalskiSoltan1} of Skalski and So{\l}tan where they consider C*-algebraic quantum group
structures on Podle\'{s}' families of quantum invertible maps. We delegate study of relations between the results of the present paper
and those of \cite{SkalskiSoltan1} to a future work.

\textbf{Notations \& Preliminaries.}
The identity map on a set $X$ is denoted by $\i_X$.
Throughout all rings have unit and ring homomorphisms preserve units. $\m{K}$ denotes
a fixed non-zero commutative ring. $\ot$ denotes tensor product over $\m{K}$. For any permutation $\sigma$ of $\{1,\cdots,n\}$,
$F_\sigma$ denotes the flip homomorphism $x_1\ot\cdots\ot x_n\mapsto x_{\sigma(1)}\ot\cdots\ot x_{\sigma(n)}$
on a tensor product $M_1\ot\cdots\ot M_n$ of $\m{K}$-modules; when there is no risk of confusion we will write $F$ instead of $F_\sigma$.
An algebra is a ring together with a structural ring homomorphism $\m{K}\to A$.
A morphism between algebras is a ring homomorphism which commutes with structural homomorphisms.
$\m{K}[x_i]_{i\in I}$ and $\m{K}^\r{c}[x_i]_{i\in I}$ denote, respectively, the polynomial algebra in non-commuting and commuting
variables $\{x_i\}_{i\in I}$. Let $B$ be an algebra. If $S\subseteq B$ generates $B$ as an algebra, we write
$I_S$ for kernel of the canonical morphism $\m{K}[x_s]_{s\in S}\to B$ defined by $x_s\mapsto s$.
$\b{Alg}$ denotes the category of algebras and $\b{Alg}(B,A)$ is the set of morphisms from $B$ to $A$.
The elements of $\b{Alg}(B,\m{K})$ are called characters of $B$. The full subcategory of commutative algebras is denoted by $\b{Alg}^\r{c}$.
The multiplication of $A$ is denoted by $\mu_A:A\ot A\to A$. Note that if $A$ is commutative then $\mu_A$ is a morphism.
By a finite-rank algebra we mean an algebra which is a free $\m{K}$-module of finite rank.
Throughout \emph{finite-rank} is abbreviated to \emph{FR}. A comultiplication on an algebra $B$ is a
morphism $\Delta:B\to B\ot B$ which is coassociative: $(\Delta\ot\i_B)\Delta=(\i_B\ot\Delta)\Delta$. $\Delta$ is cocommutative if
$\Delta=F\Delta$. A character $\epsilon$ on $B$ is a
counit if $(\i_B\ot\epsilon)\Delta=(\epsilon\ot\i_B)\Delta=\i_B$. A morphism $S:B\to B^\r{op}$ is an antipode if
$\mu_B(S\ot\i_B)\Delta=\mu_B(\i_B\ot S)\Delta=\epsilon1$. Then $(B,\Delta,\epsilon)$ and $(B,\Delta,\epsilon,S)$
are called bialgebra and Hopf-algebra, respectively.
By a (right) $B$-comodule we mean an algebra $V$ together with a right coaction of $B$, i.e. a morphism
$\rho:V\to V\ot B$ satisfying $(\rho\ot\i_B)\rho=(\i_V\ot\Delta)\rho$ and $(\i_V\ot\epsilon)\rho=\i_V$.
\section{Algebraic family of all morphisms}
In this section we consider the notions of \emph{algebraic family of morphisms} and \emph{algebraic family of all morphisms} from an algebra $B$
to another algebra $A$. We shall show that the latter exists when $A$ is a FR algebra. Also we consider some functorial
properties of these notions and some simple explicit examples and computations.
\begin{definition}\label{D2.1}
Let $A$, $B$, $B'$, $C$ and $C'$ be algebras.
\begin{enumerate}
\item[(a)] A morphism $\psi:B\to A\ot C$ is called an algebraic family of morphisms from $B$ to $A$ with parameter-algebra $C$.
When there are no ambiguity about $A$, $B$ and $C$, we call $\psi$ just 'family'.
\item[(b)] A family with parameter-algebra $0$ is called empty family of morphisms.
If $f_1,\cdots,f_n$ is a finite collection of morphisms from $B$ to $A$ then $\psi:B\to A\ot\m{K}^n$,
defined by $\psi(b)=\sum_{i=1}^nf_i(b)\ot e_i$ where $e_1,\cdots,e_n$ denote standard basis, is called trivial family.
\item[(c)] $\c{M}(B,A)$ denotes the class of all families of morphisms from $B$ to $A$. $\c{M}(B,A)$ can be considered as a
category such that objects are families and morphisms from $\psi:B\to A\ot C$ to $\psi':B\to A\ot C'$ are algebra morphisms
$\varphi:C\to C'$ satisfying $\psi'=(\i_A\ot\varphi)\psi$.
\item[(d)] Let $\psi:B\to A\ot C$ and $\psi':B'\to B\ot C'$ be families of morphisms from $B$ to $A$ and from $B'$ to $B$, respectively.
Then the family $(\psi\ot\i_{C'})\psi'$ of morphisms from $B'$ to $A$ (with parameter algebra $C\ot C'$)
is denoted by $\psi\circ\psi'$ and called composition of $\psi$ and $\psi'$.
(Note that composition of families can be considered as a functor from $\c{M}(B,A)\times\c{M}({B'},B)$ to $\c{M}({B'},A)$.)
\item[(e)] Let $\c{C}$ be a subclass of $\c{M}(B,A)$. $\psi\in\c{C}$ is called a $\c{C}$-universal family if it is an initial object of
$\c{C}$ when $\c{C}$ is considered as a full subcategory of $\c{M}(B,A)$.
(Note that if a $\c{C}$-universal family exists then it is unique up to isomorphism.)
\item[(f)] Let $\c{C}$ be a subclass of $\c{M}(B,A)$. Then $\c{C}^\r{c}$ denotes the subclass of $\c{C}$
containing families with commutative parameter-algebras.
\end{enumerate}
\end{definition}
\begin{lemma}\label{L2.0}
Let $\c{C}\subseteq\c{M}(B,A)$. Suppose that $\phi:B\to A\ot C$ is a $\c{C}$-universal family.
Consider the morphism $\psi=(\i_A\ot\varphi)\phi:B\to A\ot C/[C,C]$, where $[C,C]$ denotes commutator ideal of $C$ and
$\varphi:C\to C/[C,C]$ is the canonical projection. Then $\psi$ is $\c{C}^\r{c}$-universal if it belongs to $\c{C}$.
\end{lemma}
\begin{proof}
Straightforward.
\end{proof}
\begin{proposition}\label{P2.1}
Let $A$ and $B$ be algebras such that $A$ is FR. Then there exists a $\c{M}(B,A)$-universal family, say
$\phi:B\to A\ot Z$. This family has the following additional properties.
\begin{enumerate}
\item[(a)] The set $\{(\alpha\ot\i_Z)\phi(b)\}$, where $b\in B$
and $\alpha$ runs over all module homomorphisms from $A$ to $\m{K}$, generates $Z$ as an algebra.
\item[(b)] If $B$ is finitely generated (resp. presented) algebra then $Z$  is finitely generated (resp. presented) algebra.
\item[(c)] For any algebra $C$ there is a canonical
one-to-one correspondence between morphisms from $B$ to $A\ot C$ and morphisms from $Z$  to $C$.
In particular there is a one-to-one correspondence between morphisms from $B$ to $A$ and characters on $Z$.
\item[(d)] $Z$ represents the covariant functor $\b{Alg}(B,A\ot?)$ from $\b{Alg}$ to the category of sets.
\end{enumerate}
\end{proposition}
$\phi$ is called \emph{algebraic family of all morphisms from $B$ to $A$}. We will also use the symbols
$\f{m}(B,A)$ and $\f{M}(B,A)$ for $\phi$ and $Z$, respectively.
\begin{proof}
Suppose that $S\subseteq B$ generates $B$ as an algebra
and $\{a_i:i=1,\cdots,n\}$ is a basis for $A$ as free module. Let
$Z$ be the universal algebra generated by $z_{si}$ ($s\in S$) such that
for every polynomial $p\in I_S$ the equation $p(\{\sum_{i=1}^na_i\ot z_{si}\}_{s\in S})=0$ (in $A\ot Z$)
is satisfied. Then the assignment $s\mapsto\sum_{i=1}^na_i\ot z_{si}$ defines a morphism $\phi$.
Let $\psi:B\to A\ot C$ be a morphism and suppose that $c_{s,i}\in C$ is such that $\psi(s)=\sum_{i=1}^na_i\ot c_{si}$.
Thus the equation $p(\{\sum_{i=1}^na_i\ot c_{si}\}_{s\in S})=0$ is satisfied for every $p\in I_S$.
Then the universality of $Z$ shows that the assignment $z_{si}\mapsto c_{si}$ defines a morphism
$\bar{\psi}$ and this morphism is the only one whose satisfies $\psi=(\i_A\ot\bar{\psi})\phi$.
This shows that $\phi$ is a universal family in $\c{M}(B,A)$.
(a) and (b) immediately  follow from the construction of $Z$, and it is clear that the assignment $\psi\mapsto\bar{\psi}$ defines
the desired correspondence in (c). Also (d) trivially follows from (c).
\end{proof}
\begin{remark}\label{R2.1}
Let $A$ and $B$ be algebras such that $A$ is FR. It follows from Proposition \ref{P2.1} and Lemma \ref{L2.0}
that there exists a family  $\f{m}^\r{c}:B\to A\ot\f{M}^\r{c}(B,A)$ which is $\c{M}^\r{c}(B,A)$-universal. It also follows that
$\f{M}^\r{c}(B,A)$ is canonically isomorphic to the quotient of $\f{M}(B,A)$ by its commutator ideal, and analogs of conditions (a)-(d) of Proposition
\ref{P2.1} are satisfied for $\f{M}^\r{c}(B,A)$.
\end{remark}
The universality of families of all morphisms easily implies the following proposition.
\begin{proposition}\label{P2.2}
Let $A$ and $B$ be algebras such that $A$ is FR and commutative.
Then $\f{M}^\r{c}(B,A)$ and $\f{M}^\r{c}(B/[B,B],A)$ are isomorphic.
\end{proposition}
\begin{proof}
Straightforward.
\end{proof}
\begin{remark}\label{R2.2}
Let $A$ and $B$ be algebras such that $A$ is FR.
\begin{enumerate}
\item[(a)] It follows easily from the universal property that
$\f{M}(B,\m{K})=B$, $\f{M}^\r{c}(B,\m{K})=B/[B,B]$,
and $\f{M}(\m{K},A)=\f{M}^\r{c}(\m{K},A)=\m{K}$.
\item[(b)] From the construction of universal parameter-algebra
in the proof of Proposition \ref{P2.1} we see that $\f{M}(B,\m{K}^n)$ is the free product
$B\ast\cdots\ast B$ of $n$ copies of $B$. (Here free product means coproduct in $\b{Alg}$.)
Similarly, we have $\f{M}^\r{c}(B,\m{K}^n)=\ot_{i=1}^nB/[B,B]$ by Remark \ref{R2.1}.
It follows that $\f{M}(\m{K}[x],\m{K}^n)=\m{K}[x_1,\cdots,x_n]$ and $\f{M}^\r{c}(\m{K}[x],\m{K}^n)=\m{K}^\r{c}[x_1,\cdots,x_n]$. Also
if $\m{K}$ is an algebraically closed field and $B$ is the algebra of polynomial
functions on an affine variety $V$ then $\f{M}^\r{c}(B,\m{K}^n)$
is the algebra of polynomial functions on $V^n$.
\item[(c)] From Proposition \ref{P2.1} (c) it follows
that if there is a morphism from $B$ to $A$ then $Z:=\f{M}(B,A)\neq0$ and also $Z^\r{c}:=\f{M}^\r{c}(B,A)\neq0$.
Moreover, if $\m{K}$ is a field, it follows from Dedekind's lemma that $\dim_\m{K}Z^\r{c}$ (and so $\dim_\m{K}Z$) is not less than the cardinality
of the set of algebra morphisms from $B$ to $A$.
\item[(d)] Conversely, suppose that $Z^\r{c}\neq0$. So $Z^\r{c}$ has a maximal ideal.
If $\m{K}$ is an algebraically closed field and $B$ is finitely generated then Proposition \ref{P2.1} (b) implies that
$Z^\r{c}$ is finitely generated and therefore, it follows from Zariski's lemma that $Z^\r{c}$
has a character and so there exits an algebra morphism from $B$ to $A$.
\item[(e)] Suppose that $\m{K}$ is an algebraically closed field and $B$ is finitely generated. Moreover, suppose that
there are only finitely many morphisms from $B$ to $A$. Then $Z^\r{c}/J$ is isomorphic to the algebra $\m{K}^n$
where $J$ denotes the Jacobson radical and $n$ is the number of distinct morphisms from $B$ to $A$. To see this fact,
let $f_1,\cdots,f_n$ be the collection of all morphisms from $B$ to $A$ and let
$\psi:B\to A\ot\m{K}^n$ be its trivial family. Let $\bar{\psi}:Z^\r{c}\to\m{K}^n$ be the unique morphism which satisfies
$\psi=(\i_A\ot\bar{\psi})\f{m}^\r{c}$. From Proposition \ref{P2.1} (c) it follows that $Z^\r{c}$ has exactly $n$ distinct characters
and these are of the form $p_i\bar{\psi}$ ($i=1,\cdots,n$) where $p_i:\m{K}^n\to\m{K}$ denotes the canonical projection on the $i$'th component.
It follows that any maximal ideal of $Z^\r{c}$ is in the form of $\ker(p_i\bar{\psi})$ and so $\bar{\psi}$ induces an injective morphism
from $Z^\r{c}/J$ to $\m{K}^n$. From Dedekind's Lemma it follows that $\dim_\m{K}Z^\r{c}/J\geq n$. So $Z^\r{c}/J$ and $\m{K}^n$ are isomorphic.
\end{enumerate}
\end{remark}
Let $A$ be a FR algebra. Let $B_1$ and $B_2$ be two algebras and $f:B_1\to B_2$
be a morphism. Then the universality of $\f{M}(B_1,A)$ shows that there is a unique morphism $\f{M}(f,A)=\f{M}(f)$
such that makes the following diagram commutative.
$$\xymatrix{B_1\ar[d]_{f}\ar[r]^-{\f{m}(B_1,A)}& A\ot\f{M}(B_1,A)\ar[d]^{\i_A{\ot}\f{M}(f)}\\
B_2\ar[r]^-{\f{m}(B_2,A)}&A\ot \f{M}(B_2,A)}$$
Analogously, if $g:B_2\to B_3$ is another morphism we find morphisms $\f{M}(g):\f{M}(B_2,A)\to\f{M}(B_3,A)$ and
$\f{M}(gf):\f{M}(B_1,A)\to\f{M}(B_3,A)$. Then again universality of $\f{M}(B_1,A)$ implies $\f{M}(gf)=\f{M}(g)\f{M}(f)$.
Similarly, we have $\f{M}(\i_B)=\i_{\f{M}(B,A)}$ for any algebra $B$.
Analogous statements are also satisfied for $\f{M}^\r{c}$. So we have proved the theorem below.
\begin{theorem}\label{T2.1}
Let $A$ be a FR algebra. Then $\f{M}(?,A):\b{Alg}\to\b{Alg}$ (resp. $\f{M}^\r{c}(?,A):\b{Alg}\to\b{Alg}^\r{c}$)
is a covariant functor.
\end{theorem}
Note that for a fixed algebra $B$, $\f{M}(B,?)$ and $\f{M}^\r{c}(B,?)$ can be considered as contravariant functors from category
of FR algebras to $\b{Alg}$ and $\b{Alg}^\r{c}$, respectively.
\begin{theorem}\label{T2.2}
Let $A$ be a FR commutative algebra. Then  $\f{M}^\r{c}(?,A)$ preserves tensor product.
\end{theorem}
\begin{proof}
Suppose that $B_1$ and $B_2$ are two algebras. Let $C_i=\f{M}^\r{c}(B_i,A)$ and $\phi_i=\f{m}^\r{c}(B_i,A)$ ($i=1,2$) and also let
$D=\f{M}^\r{c}(B_1\ot B_2,A)$ and $\psi=\f{m}^\r{c}(B_1\ot B_2,A)$. We must show that $C_1\ot C_2$ and $D$ are isomorphic.
Let $\alpha_i:B_i\to B_1\ot B_2$ be the structural morphisms, i.e. $\alpha_1(b_1)=b_1\ot1$ and $\alpha_2(b_2)=1\ot b_2$ ($b_1\in B_1,b_2\in B_2$).
Then we find morphisms $\f{M}(\alpha_i):C_i\to D$. The coproduct structure of $C_1\ot C_2$ (in $\b{Alg}^\r{c}$) induces a morphism
$f:C_1\ot C_2\to D$ such that $f\gamma_i=\f{M}(\alpha_i)$ where $\gamma_i:C_i\to C_1\ot C_2$ are structural morphisms.
Let $g:D\to C_1\ot C_2$ be the unique morphism that makes the following diagram commutative.
$$\xymatrix{B_1\ot B_2\ar[dd]^{\phi{\ot}\phi}\ar[rr]_-{\psi}&& A\ot D\ar[d]_{\i_A{\ot}g}\\
&& A\ot (C_1\ot C_2)\\
A\ot C_1\ot A\ot C_2\ar[rr]^-{\i_A\ot F\ot \i_{C_2}}&&(A\ot A)\ot (C_1\ot C_2)\ar[u]^{\mu_A{\ot} \i_{C_1{\ot} C_2}}}$$
Then it is easily checked that $gf\gamma_i=\gamma_i$. This fact together with
the coproduct universal property of $C_1\ot C_2$ implies $gf=\i_{(C_1\ot C_2)}$.
Also it is easily checked that $(\i_A\ot fg)\psi=\psi$. So by the universal property of $D$
we have $fg=\i_D$. This completes the proof.
\end{proof}
\begin{remark}\label{R2.3}
Let $B_1$, $B_2$ and $A$ be algebras such that $A$ is commutative and FR. Then,
there is a canonical morphism from $\f{M}(B_1\ot B_2,A)$ to $\f{M}(B_1,A)\ot\f{M}(B_2,A)$, analogous to $g$ in the proof of Theorem \ref{T2.2}.
\end{remark}
The following theorem states an \emph{exponential law} in our dual formalism.
\begin{theorem}\label{T2.3}
Let $A_1$, $A_2$ and $B$ be algebras such that $A_1$ and $A_2$ are FR. Then the algebras
$\f{M}(B,A_1\ot A_2)$ and $\f{M}(\f{M}(B,A_1),A_2)$ are isomorphic.
\end{theorem}
\begin{proof}
Let $\phi_i=\f{m}(B,A_i)$, $\phi_{12}=\f{m}(B,A_1\ot A_2)$ and $\phi=\f{m}(\f{M}(B,A_1),A_2)$.
Let $\psi$ be the unique morphism such that the diagram
\begin{equation}\label{Dig2.1}
\xymatrix{B\ar[d]_{\phi_1}\ar[rr]^-{\phi_{12}}&& ({A_1}\ot {A_2})\ot\f{M}(B,{A_1}\ot {A_2})\ar[d]^{\i_{({A_1}\ot {A_2})}{\ot}\psi}\\
{A_1}\ot\f{M}(B,{A_1})\ar[rr]^-{\i_{A_1}\ot\phi}&&({A_1}\ot {A_2})\ot\f{M}(\f{M}(B,{A_1}),{A_2})}
\end{equation}
is commutative. Let $\varphi$ and $\psi'$ be morphisms that make the following diagrams commutative.
\begin{equation}\label{Dig2.2}
\xymatrix{B\ar[1,1]^{\phi_{12}}\ar[r]^-{\phi_1}& {A_1}\ot\f{M}(B,{A_1})\ar[d]^{\i_{A_1}{\ot}\varphi}\\
&{A_1}\ot({A_2}\ot\f{M}(B,{A_1}\ot {A_2}))}
\end{equation}
\begin{equation}\label{Dig2.3}
\xymatrix{\f{M}(B,{A_1})\ar[1,1]^{\varphi}\ar[r]^-{\phi}&{A_2}\ot\f{M}(\f{M}(B,{A_1}),{A_2})\ar[d]^{\i_{A_2}\ot\psi'}\\
&{A_2}\ot\f{M}(B,{A_1}\ot {A_2})}
\end{equation}
by (\ref{Dig2.3}), we have $\i_{A_1}\ot\varphi=(\i_{(A_1\ot A_2)}\ot\psi')(\i_{A_1}\ot\phi)$ and so by (\ref{Dig2.1}) and (\ref{Dig2.2}) we have,
\begin{equation*}
\begin{split}
(\i_{(A_1\ot A_2)}\ot(\psi'\psi))\phi_{12}&=(\i_{(A_1\ot A_2)}\ot\psi')(\i_{(A_1\ot A_2)}\ot\psi)\phi_{12}\\
&=(\i_{(A_1\ot A_2)}\ot\psi')(\i_{A_1}\ot\phi)\phi_1\\
&=(\i_{A_1}\ot\varphi)\phi_1\\
&=\phi_{12}.
\end{split}
\end{equation*}
This equality together with universality of $\f{M}(B,A_1\ot A_2)$ imply that $\psi'\psi$ is the identity morphism.
Let $b\in B$ and $f:A_1\to\m{K}$ be a module homomorphism and let $c=(f\ot\i_{\f{M}(B,{A_1})})\phi_1(b)$ (note that
by Proposition \ref{P2.1}(a) such elements generates $\f{M}(B,A_1)$). By (\ref{Dig2.2}) and (\ref{Dig2.1}) we have,
\begin{equation*}
\begin{split}
(\i_{A_2}\ot\psi)\varphi(c)&=(\i_{A_2}\ot\psi)\varphi(f\ot\i_{\f{M}(B,{A_1})})\phi_1(b)\\
&=(\i_{A_2}\ot\psi)(f\ot\i_{({A_2}\ot\f{M}(B,{A_1}\ot {A_2}))})\phi_{12}(b)\\
&=(f\ot\i_{({A_2}\ot\f{M}(\f{M}(B,{A_1}),{A_2}))})(\i_{A_1}\ot\phi)\phi_1(b)\\
&=\phi(f\ot\i_{\f{M}(B,{A_1})})\phi_1(b)\\
&=\phi(c).
\end{split}
\end{equation*}
Thus $(\i_{A_2}\ot\psi)\varphi=\phi$. This equality together with (\ref{Dig2.3}) imply,
\begin{equation*}
\begin{split}
(\i_{A_2}\ot(\psi\psi'))\phi&=(\i_{A_2}\ot\psi)(\i_{A_2}\ot\psi')\phi\\
&=(\i_{A_2}\ot\psi)\varphi\\
&=\phi
\end{split}
\end{equation*}
Then the above equality together with universality of $\f{M}(\f{M}(B,A_1),A_2)$  show that $\psi\psi'$ is the identity morphism.
This completes the proof.
\end{proof}
Analogue of Theorem \ref{T2.3} is also satisfied for the functor $\f{M}^\r{c}$. The following lemma states a fact that will be used
in Section 5.
\begin{lemma}\label{L2.1}
Let $A$ and $B$ be algebras such that $A$ is commutative and FR. Then there is a canonical isomorphism
$\f{M}(B^\r{op},A)\cong\f{M}(B,A)^\r{op}$.
\end{lemma}
\begin{proof}
Straightforward.
\end{proof}
\section{Algebraic families of isomorphisms}
In this section we define \emph{algebraic families of isomorphisms} and \emph{algebraic families of all isomorphisms}.
\begin{definition}\label{D2.3}
Let $A$ and $B$ be algebras and let $\psi:B\to A\ot C$ be a family of morphisms from $B$ to $A$.
Suppose that there exists a family $\phi:A\to B\ot C$ such that
$(\i_B\ot\mu_C)(\phi\circ\psi)=\i_B\ot1$ and $(\i_A\ot\mu_C)(\psi\circ\phi)=\i_A\ot1$.
Then $\psi$ is called a family of isomorphisms from $B$ onto $A$ or, an invertible family of morphisms from $B$ to $A$.
Also, $\phi$ is called an inverse for $\psi$. The class of all invertible families of morphisms from $B$ to $A$ is denoted by $\c{I}(B,A)$.
\end{definition}
Suppose that $\psi:B\to A\ot C$ is an invertible family and let $\phi_1,\phi_2:A\to B\ot C$ be two inverses for $\psi$.
Then from the identity $(\i_A\ot\mu_C)(\psi\ot\i_C)\phi_2(a)=a\ot1$ ($a\in A$) it is easily seen that
$\phi_1=(\i_B\ot\mu_C)(\phi_1\ot\i_C)(\i_A\ot\mu_C)(\psi\ot\i_C)\phi_2$.
Also it follows from associativity of $\mu_C$ (i.e. $\mu_C(\i_C\ot\mu_C)=\mu_C(\mu_C\ot\i_C)$) and identity
$(\i_B\ot\mu_C)(\phi_1\ot\i_C)\psi(b)=b\ot1$ ($b\in B$) that
$\i_{B\ot C}=(\i_B\ot\mu_C)(\phi_1\ot\i_C)(\i_A\ot\mu_C)(\psi\ot\i_C)$.
Thus we conclude $\phi_1=\phi_2$. So we denote the inverse of an invertible family $\psi$ by $\psi^{-1}$.
It is also easily checked that the composition of $\psi$ with another invertible family $\psi':B'\to B\ot C'$
is an invertible family, too. Indeed, $(\psi\circ\psi')^{-1}=(\i_{B'}\ot F)({\psi'}^{-1}\circ\psi^{-1})$.
\begin{example}\label{E3.1}
\begin{enumerate}
\item[(a)] Suppose that $f_1,\cdots,f_n$ is a finite collection of isomorphisms from $B$ onto $A$. Then the trivial family
$\psi:B\to A\ot\m{K}^n$, defined by $\psi(b)=\sum_{i=1}^nf_i(b)\ot e_i$ ($b\in B$), is an invertible family. Indeed,
$\psi^{-1}(a)=\sum_{i=1}^nf^{-1}_i(a)\ot e_i$ ($a\in A$). Also, note that the empty family of morphisms is an invertible family.
\item[(b)] Let $(B,\Delta,\epsilon,S)$ be a commutative Hopf-algebra and $V$ be a $B$-comodule with the coaction $\rho:V\to V\ot B$.
Let $\rho'=(\i_V\ot S)\rho$. We have
\begin{equation*}
\begin{split}
(\i_V\ot\mu_B)(\rho'\circ\rho)&=(\i_V\ot\mu_B)([(\i_V\ot S)\rho]\ot\i_B)\rho\\
&=(\i_V\ot\mu_B)(\i_V\ot S\ot\i_B)(\rho\ot\i_B)\rho\\
&=(\i_V\ot\mu_B)(\i_V\ot S\ot\i_B)(\i_V\ot\Delta)\rho\\
&=(\i_V\ot[\mu_B(S\ot\i_B)\Delta])\rho\\
&=(\i_V\ot\epsilon1_B)\rho\\
&=\i_V\ot1_B.
\end{split}
\end{equation*}
Analogously, we have $(\i_V\ot\mu_B)(\rho\circ\rho')=\i_V\ot1_B$. Therefore, $\rho$ is an invertible family with inverse $\rho'$.
\end{enumerate}
\end{example}
\begin{proposition}\label{P2.3}
Let $A$ and $B$ be FR algebras. Then there exists an $\c{I}(B,A)$-universal family, say $\phi:B\to A\ot Z$. Moreover,
this family has the following additional properties.
\begin{enumerate}
\item[(a)] $Z$ is a finitely generated algebra.
\item[(b)] There is a canonical one-to-one correspondence between isomorphisms from $B$ onto $A$ and characters of $Z$.
\end{enumerate}
Analogues of the above statements are also satisfied for the subclass $\c{I}^\r{c}(B,A)$.
Indeed, the canonical morphism $\phi^\r{c}:B\to A\ot Z/[Z,Z]$, induced by $\phi$, is an $\c{I}^\r{c}(B,A)$-universal family.
\end{proposition}
We call both of $\phi$ and $\phi^\r{c}$ \emph{algebraic family of all isomorphisms from $B$ onto $A$}, and sometimes denote them
by $\f{i}(B,A)$ and $\f{i}^\r{c}(B,A)$, and their parameter-algebras by $\f{I}(B,A)$ and $\f{I}^\r{c}(B,A)$, respectively.
\begin{proof}
Suppose that $A_0=\{a_i:i=1,\cdots,a_m\}$ (resp. $B_0=\{b_j:j=1,\cdots,b_n\}$) is a basis for $A$ (resp. $B$) as free-module.
Let $Z$ be the universal algebra generated by $z_{ji}$ and $z'_{ij}$ such that the following four classes of equations are satisfied.
\begin{enumerate}
\item[(i)] $q(\sum_{i=1}^ma_i\ot z_{1i},\cdots,\sum_{i=1}^ma_i\ot z_{ni})=0$ for every $q\in{I}_{B_0}$.
\item[(ii)] $p(\sum_{j=1}^nb_j\ot {z'}_{1j},\cdots,\sum_{j=1}^nb_j\ot {z'}_{mj})=0$ for every $p\in{I}_{A_0}$.
\item[(iii)] $\sum_{i=1}^m{z'}_{ik}z_{ji}=0$ and $\sum_{i=1}^m{z'}_{ij}z_{ji}=1$ for $1\leq j,k\leq n$ with $j\neq k$.
\item[(iv)] $\sum_{j=1}^nz_{jl}{z'}_{ij}=0$ and $\sum_{j=1}^nz_{ji}{z'}_{ij}=1$ for $1\leq i,l\leq m$ with $i\neq l$.
\end{enumerate}
(i) shows that the assignment $b_j\mapsto\sum_{i=1}^ma_i\ot z_{ji}$ defines a morphism
$\phi:B\to A\ot Z$. Similarly, (ii) shows that the assignment $a_i\mapsto\sum_{j=1}^nb_j\ot{z'}_{ij}$
defines a morphism $\phi':A\to B\ot Z$. (iii) shows that $(\i_B\ot\mu_Z)(\phi'\ot\i_Z)\phi(b_j)=b_j\ot1$.
Similarly, (iv) shows that $(\i_A\ot\mu_Z)(\phi\ot\i_Z)\phi'(a_i)=a_i\ot1$. So $\phi$ is a family of isomorphisms from $B$ onto $A$
with inverse $\phi'$. Now suppose that $\psi:B\to A\ot C$ is a family of isomorphisms and let $c_{ji}$ and ${c'}_{ij}$
in $C$ be such that $\psi(b_j)=\sum_{i=1}^ma_i\ot c_{ji}$ and $\psi^{-1}(a_i)=\sum_{j=1}^nb_j\ot {c'}_{ij}$.
Then, the analogs of the above four classes of equations
are satisfied for $c_{ji}$ and ${c'}_{ij}$. Thus, the universality of $Z$ shows that
the assignments $z_{ji}\mapsto c_{ji}$ and ${z'}_{ij}\mapsto{c'}_{ij}$ define an algebra morphism $\bar{\psi}:Z\to C$
which satisfies $\psi=(\i_A\ot\bar{\psi})\phi$. Uniqueness of $\bar{\psi}$ also follows from universality of $Z$.
Therefore $\phi$ is a universal family of isomorphisms. (a) is trivial by the construction of $Z$. Let $\psi:B\to A$ be an isomorphism. Then
$\psi$ can be considered as a family of isomorphisms with the parameter-algebra $\m{K}$, and so,
$\bar{\psi}$ is a character on $Z$. This proves (b).
The other statements follow from Lemma \ref{L2.0}.
\end{proof}
\begin{remark}
Let $A$ and $B$ be FR algebras.
\begin{enumerate}
\item[(a)] If follows from Proposition \ref{P2.3} (b)
that if there is an isomorphism from $B$ onto $A$ then $\f{I}(B,A)\neq0$ and
also $\f{I}^\r{c}(B,A)\neq0$.
\item[(d)] Conversely, if $\m{K}$ is an algebraically closed field, then from
$\f{I}^\r{c}(B,A)\neq0$ it follows that $B$ and $A$ are isomorphic (see Remark \ref{R2.2} (d)).
\end{enumerate}
\end{remark}
Let us call two algebras $A$ and $B$ \emph{quasi-isomorphic} if there is a nonempty algebraic family of isomorphisms from $B$
onto $A$. (Note that this notion is very far from the notion of quasi-isomorphic chain complexes in homological algebra.)
Then, quasi-isomorphism is an equivalence relation on the class of algebras.
It is clear that isomorphism (of algebras) implies quasi-isomorphism. But there are non-isomorphic algebras which are quasi-isomorphic:
Let $C$ be an algebra which has not IBN \cite{Lam1}, that is there are natural numbers $n$ and $m$ with $n\neq m$ such that
$C^n$ and $C^m$ are isomorphic left free $C$-modules. Equivalently, there are elements $s_{ji}$ and $t_{ij}$ in $C$,
for $1\leq i\leq m$ and $1\leq j\leq n$, such that $\sum_{i=1}^ms_{ji}t_{ij'}=\delta_{jj'}$ and $\sum_{j=1}^nt_{ij}s_{ji'}=\delta_{ii'}$.
Let $A=\m{K}[x_1,\cdots,x_m]$ and $B=\m{K}[y_1,\cdots,y_n]$ and let the families $\psi:A\to B\ot C$ and $\phi:B\to A\ot C$
be defined by $\psi(x_i)=\sum_{j=1}^ny_j\ot s_{ji}$ and $\phi(y_j)=\sum_{i=1}^mx_i\ot t_{ij}$. Then it is easily checked that
$\psi=\phi^{-1}$. So, $A$ and $B$ are quasi-isomorphic. Indeed, this is the case for every $n$ and $m$; this follows from the fact that
if $C$ is the algebra of all module endomorphisms on the infinite direct sum $\oplus_{i=1}^\infty\m{K}$ then $C^n$ and $C^m$
are isomorphic for every $n$ and $m$ as left $C$-modules (see \cite[Example 1.4]{Lam1}).

We end this section with a list of interesting questions.
Is there any relation between quasi-isomorphism and Morita equivalence, or algebraic homotopy equivalence (\cite{Gersten1}, \cite{Jardine1})?
Morita equivalent algebras are characterized by existence of some specific bimodules, see \cite[Section 18C]{Lam1}. Is it possible to characterize
quasi-isomorphic algebras in this manner?
For any fixed algebra $C$, let $\b{Alg}^*$ be a category whose objects are algebras and a typical morphism from $B$ to $A$ is a family
$\phi:B\to A\ot C^{\ot n}$, for some $n\geq0$, together with the composition of families as composition law of the category. (Note that
if $C=\m{K}$ then $\b{Alg}^*=\b{Alg}$.) Is there a Quillen's model category structure (\cite{Baues1}) on $\b{Alg}^*$ such that algebraic
families of isomorphisms play the role of quasi-equivalences of the model? Indeed the term quasi-isomorphism
introduced above refers to this (conjectured) property.
\section{Some other universal families of morphisms}
Let $A$ and $B$ be fixed algebras. Consider the following four types of algebraic families.
\begin{enumerate}
\item[(1)] Let $M$ be a $\m{K}$-module and, $\beta:M\to B$ and $\alpha:M\to A$ be module homomorphisms.
$\c{M}_1$ denotes the class of all families $\psi:B\to A\ot C$ satisfying
$$\alpha(m)\ot1=\psi\beta(m)\hspace{10mm}(m\in M).$$
\item[(2)] Let $M'$ be a $\m{K}$-module and, $\beta':B\to M$ and $\alpha':A\to M$ be module homomorphisms.
$\c{M}_2$ denotes the class of all families $\psi:B\to A\ot C$ satisfying
$$\beta'(b)\ot1=(\alpha'\ot\i_{C})\psi(b)\hspace{10mm}(b\in B).$$
\item[(3)] Let $\Delta:B\to B\ot B$ and $\Gamma:A\to A\ot A$ be module homomorphisms. $\c{M}_3$ denotes the class
of all families $\psi:B\to A\ot C$ satisfying
$$(\Gamma\ot\i_C)\psi=(\i_{A\ot A}\ot\mu_C)(\i_A\ot F\ot\i_C)(\psi\ot\psi)\Delta.$$
\item[(4)] Let $N$ be a $\m{K}$-module and, $\Lambda:B\to B\ot N$ and $\Theta:A\to A\ot N$ be module homomorphisms. $\c{M}_4$
denotes  the class of all families $\psi:B\to A\ot C$ satisfying
$$(\i_A\ot F)(\Theta\ot\i_C)\psi=(\psi\ot\i_N)\Lambda.$$
\end{enumerate}
The algebraic families introduced in (1)-(4) are respectively analogues of the following four classes of classical families of maps.
\begin{enumerate}
\item[(1')] Families of bundle morphisms between two fibre bundles over a classical space.
\item[(2')] Families of continuous maps between two classical spaces which are identity over a common subspace.
\item[(3')] Families of continuous group homomorphisms between two topological groups.
\item[(4')] Families of $G$-maps between two $G$-spaces for a topological group $G$.
\end{enumerate}
All proofs of theorems below are analogous to the proofs of Propositions \ref{P2.1} and \ref{P2.3}, and are omitted for brevity.
\begin{theorem}\label{T3.1}
Together with the assumptions of (1) suppose that $A$ is FR. Then there exists an
$\c{M}_1$-universal family, denoted by ${\phi_1}:B\to A\ot {Z_1}$. Moreover,
there is a canonical one-to-one correspondence
between characters of ${Z_1}$ and morphisms $f:B\to A$ satisfying $f\beta=\alpha$.
\end{theorem}
\begin{theorem}\label{T3.2}
Together with the assumptions of (2) suppose that $A$ is FR and $M'$ is free as module.
Then there is an $\c{M}_2$-universal family, denoted by ${\phi_2}:B\to A\ot {Z_2}$. Moreover,
there is a canonical one-to-one correspondence
between characters of ${Z_2}$ and morphisms $f:B\to A$ satisfying $\beta'=\alpha'f$.
\end{theorem}
\begin{theorem}\label{T3.3}
Together with the assumptions of (3) suppose that $B$ is free as module and $A$ is FR.
Then there is an $\c{M}_3$-universal family, denoted by ${\phi_3}:B\to A\ot {Z_3}$. Moreover,
there is a canonical one-to-one correspondence
between characters of ${Z_3}$ and morphisms $f:B\to A$ satisfying $(f\ot f)\Delta=\Gamma f$.
\end{theorem}
\begin{theorem}\label{T3.4}
Together with the assumptions of (4) suppose that $B$ and $N$ are free as module and $A$ is FR.
Then there is an $\c{M}_4$-universal family, denoted by ${\phi_4}:B\to A\ot {Z_4}$. Moreover,
there is a canonical one-to-one correspondence
between characters of ${Z_4}$ and morphisms $f:B\to A$ satisfying $(f\ot\i_N)\Lambda=\Theta f$.
\end{theorem}
\begin{remark}\label{R3.1}
Under the assumptions of Theorem 4.i ($i=1,\cdots,4$), it follows from Lemma \ref{L2.0}
that there exists an $\c{M}^\r{c}_i$-universal family. Moreover, the parameter-algebra $Z_i^\r{c}$ of that family is
canonically isomorphic to $Z_i/[Z_i,Z_i]$. Also, if there is a morphism $f:B\to A$ which satisfies the condition of Theorem 4.i,
then $Z_i^\r{c}\neq0$.
\end{remark}
Let $\c{I}_i$ ($i=1,\cdots,4$) denote the subclass of $\c{M}_i$ containing invertible families.
\begin{theorem}\label{T3.5}
Suppose that $A$ and $B$ are FR and the assumptions of Theorem 4.i ($i=1,\cdots,4$) are satisfied. Then
$\c{I}_i$ and $\c{I}^\r{c}_i$ have universal families.
\end{theorem}
In the next sections, we show that the constructions achieved in this section may be endowed with Hopf-algebra structures in somewhat
natural manners.
\section{The action of $\f{M}(?,A)$ on Hopf-algebras}
This section is devoted to construction of a Hopf-algebra structure on parameter-algebras of families of morphisms
from FR commutative algebras to Hopf-algebras. Most of the results in this section are purely algebraic analogs of the results of
\cite{Soltan3} and \cite{Sadr1}. By using the fact that any commutative finite dimensional C*-algebra is isomorphic to a finite product
$\m{C}\oplus\cdots\oplus\m{C}$, So{\l}tan in \cite{Soltan3} has showed that this construction in the C*-case is a special case of the Wang
free product of compact quantum groups \cite{Wang2}, see also Remark \ref{R4.1} below.

Let $A$ be a FR commutative algebra and $B$ be an arbitrary algebra.  Let $\phi=\f{m}(B,A)$ and
$C=\f{M}(B,A)$. Suppose that $\Delta$ is a comultiplication on $B$
and let $\bar{\Delta}:C\to C\ot C$ be the composition of $\f{M}(\Delta)$ with the canonical morphism
from $\f{M}(B\ot B,A)$ to $C\ot C$ (see Remark \ref{R2.3}). More explicitly, $\bar{\Delta}$ is the
unique morphism that makes the following diagram commutative.
\begin{equation*}\label{d1}
\xymatrix{B\ar[rr]_-{\phi}\ar[d]^{\Delta}&& A\ot C\ar[d]_{\i_A{\ot}\bar{\Delta}}\\
B\ot B\ar[d]^{\phi{\ot}\phi}&& A\ot (C\ot C)\\
A\ot C\ot A\ot C\ar[rr]^-{\i_A\ot F\ot \i_C}&&(A\ot A)\ot (C\ot C)\ar[u]^{\mu_A{\ot} \i_{C{\ot} C}}}
\end{equation*}
\begin{lemma}\label{L5.1}
$\bar{\Delta}$ is a comultiplication on $C$.
\end{lemma}
\begin{proof}
By the universality of $C$, it is enough to show that equality
$(\i_A{\ot}\i_C{\ot}\bar{\Delta}) (\i_A{\ot}\bar{\Delta}) \phi=(\i_A{\ot}\bar{\Delta}{\ot}\i_C) (\i_A{\ot}\bar{\Delta}) \phi$ is satisfied.
We have,
\begin{equation}\label{E5.1}
\begin{split}
(\i_A{\ot}\i_C{\ot}\bar{\Delta}) (\i_A{\ot}\bar{\Delta}) \phi
=&(\i_A{\ot}\i_C{\ot}\bar{\Delta}) (\mu_A{\ot} \i_{C{\ot} C})F(\phi{\ot}\phi) \Delta\\
=&(\mu_A{\ot}\i_C{\ot}\bar{\Delta}) F (\phi{\ot}\phi) \Delta\\
=&(\mu_A{\ot}\i_{C{\ot}C{\ot}C})F(\i_{A{\ot}C{\ot}A}{\ot}\bar{\Delta})(\phi{\ot}\phi)\Delta\\
=&(\mu_A{\ot}\i_{C{\ot}C{\ot}C})F(\i_{A{\ot}C{\ot}A}{\ot}\bar{\Delta})(\phi{\ot}\i_{A{\ot}C})(\i_B{\ot}\phi)\Delta\\
=&(\mu_A{\ot}\i_{C{\ot}C{\ot}C})F(\phi{\ot}\i_{A{\ot}C{\ot}C})(\i_B{\ot}\i_A{\ot}\bar{\Delta})(\i_B{\ot}\phi)\Delta\\
=&(\mu_A{\ot}\i_{C{\ot}C{\ot}C})F(\phi{\ot}\i_A{\ot}\bar{\Delta})(\i_B{\ot}\phi)\Delta\\
=&(\mu_A{\ot}\i_{C{\ot}C{\ot}C})F(\phi{\ot}[(\i_A{\ot}\bar{\Delta})\phi])\Delta\\
=&(\mu_A{\ot}\i_{C{\ot}C{\ot}C})F(\phi{\ot}[(\mu_A{\ot} \i_{C{\ot} C})F (\phi{\ot}\phi) \Delta])\Delta\\
=&(\mu_A{\ot}\i_{C{\ot}C{\ot}C})F(\i_{A{\ot}C}{\ot}\mu_A{\ot} \i_{C{\ot} C})F(\phi{\ot}\phi{\ot}\phi)(\i_B{\ot}\Delta)\Delta\\
\end{split}
\end{equation}
\begin{equation}\label{E5.2}
\begin{split}
(\i_A{\ot}\bar{\Delta}{\ot}\i_C) (\i_A{\ot}\bar{\Delta}) \phi
=&(\i_A{\ot}\bar{\Delta}{\ot}\i_C)(\mu_A{\ot} \i_{C{\ot} C})F(\phi{\ot}\phi) \Delta\\
=&(\mu_A{\ot}\bar{\Delta}{\ot}\i_C)F(\phi{\ot}\phi) \Delta\\
=&(\mu_A{\ot}\i_{C{\ot}C{\ot}C})F(\i_A{\ot}\bar{\Delta}{\ot}\i_A{\ot}\i_C)(\phi{\ot}\phi) \Delta\\
=&(\mu_A{\ot}\i_{C{\ot}C{\ot}C})F([(\i_A{\ot}\bar{\Delta}) \phi]{\ot}\phi)\Delta\\
=&(\mu_A{\ot}\i_{C{\ot}C{\ot}C})F([(\mu_A{\ot} \i_{C{\ot} C})F(\phi{\ot}\phi) \Delta]{\ot}\phi)\Delta\\
=&(\mu_A{\ot}\i_{C{\ot}C{\ot}C})F(\mu_A{\ot}\i_{C{\ot}C{\ot}A{\ot}C})F(\phi{\ot}\phi{\ot}\phi)(\Delta{\ot}\i_B)\Delta\\
\end{split}
\end{equation}
It is easily checked that the morphisms
$$(\mu_A{\ot}\i_{C{\ot}C{\ot}C})F(\i_{A{\ot}C}{\ot}\mu_A{\ot} \i_{C{\ot} C})F\hspace{3mm}\text{ and }\hspace{3mm}
(\mu_A{\ot}\i_{C{\ot}C{\ot}C})F(\mu_A{\ot}\i_{C{\ot}C{\ot}A{\ot}C})F$$
appeared respectively in the last rows of (\ref{E5.1}) and of (\ref{E5.2})  are equal.
Now, we get the result by using the identity $(\i_B{\ot}\Delta)\Delta=(\Delta{\ot}\i_B)\Delta$.
\end{proof}
\begin{lemma}\label{L5.2}
$\bar\Delta$ is cocommutative if $\Delta$ is.
\end{lemma}
\begin{proof}
Straightforward.
\end{proof}
Suppose that $\epsilon:B\to\m{K}$ is a left counit for $B$.
Let $\tilde{\epsilon}:B\to A$ be defined by $b\mapsto\epsilon(b)1$
and let $\bar{\epsilon}$ be the unique character of $C$ satisfying $(\i_A\ot\bar{\epsilon})\phi=\tilde{\epsilon}$
(indeed, using notations of Theorem \ref{T2.1}, $\bar{\epsilon}=\f{M}(\epsilon)$).
\begin{lemma}\label{L5.3}
$\bar{\epsilon}$ is a left counit on $C$.
\end{lemma}
\begin{proof}
By the universal property of $C$ it is enough to show that
$(\i_A{\ot}[(\bar{\epsilon}\ot\i_C)\bar{\Delta}])\phi=\phi$. For every $b\in B$ we have,
\begin{equation*}
\begin{split}
(\i_A{\ot}[(\bar{\epsilon}\ot\i_C)\bar{\Delta}])\phi(b)=&(\i_A{\ot}\bar{\epsilon}{\ot}\i_C) (\i_A{\ot}\bar{\Delta})\phi(b)\\
=&(\i_A{\ot}\bar{\epsilon}{\ot}\i_C) (m{\ot}\i_{C{\ot}C}) (\i_A{\ot}F{\ot}\i_{C}) (\phi{\ot}\phi) \Delta(b)\\
=&(\mu_A{\ot}\bar{\epsilon}{\ot}\i_C) (\i_A{\ot}F{\ot}\i_{C}) (\phi{\ot}\phi) \Delta(b)\\
=&(\mu_A{\ot}\i_C) (\i_A{\ot}\bar{\epsilon}{\ot}\i_{A}{\ot}\i_C) (\phi{\ot}\phi) \Delta(b)\\
=&(\mu_A{\ot}\i_C) ([(\i_A\ot\bar{\epsilon})\phi]\ot\phi) \Delta(b)\\
=&(\mu_A{\ot}\i_C) (\tilde{\epsilon}\ot\phi) \Delta(b)\\
=&(\mu_A{\ot}\i_C) (\i_A\ot\phi) (\tilde{\epsilon}\ot\i_B) \Delta(b)\\
=&(\mu_A{\ot}\i_C) (\i_A\ot\phi)(1_A\ot b)\\
=&\phi(b).
\end{split}
\end{equation*}
\end{proof}
Analogue of the above result is satisfied for a right counit. In particular, if $\epsilon$ is a counit for $B$ then $\bar{\epsilon}$
is a counit for $C$.
\begin{remark}
Let $B$ and $D$ be algebras such that $D$ is commutative. Then any bialgebra structure $(B,\Delta,\epsilon)$ induces
a monoidal structure (\cite[Chapter 9]{Majid1}) $\hat{\ot}$ on the category $\c{M}(B,D)$ as follows.
For any two families $\psi_1:B\to D\ot E_1$ and $\psi_2:B\to D\ot E_2$ of morphisms from $B$ to $D$ let
$\psi_1\hat{\ot}\psi_2=(\mu_{D}\ot\i_{E_1\ot E_2})(\i_{D}\ot F\ot\i_{E_2})(\psi_1\ot\psi_2)\Delta$.
Associativity of the bifunctor $\hat{\ot}$ follows from coassociativity of $\Delta$, and the unit object of the monoidal structure
is a family $B\to D$ defined by $b\mapsto\epsilon(b)1$. If $\Delta$ is cocommutative then $\hat{\ot}$ is a symmetric
monoidal structure. Moreover, suppose that $B$ is a cocommutative Hopf-algebra. Then for any character $\xi$ of $B$, the family
$B\to D$, defined by $b\mapsto\xi(b)1$, is an invertible object of $\c{M}(B,D)$. Also, any invertible object is of this form.
It follows that the Picard group (\cite{HoveyPalmieriStrickland1},\cite{May1})
of $\c{M}(B,D)$ is isomorphic to the group of characters of $B$ (with group operation induced by $\Delta$).
\end{remark}
Suppose now that $(B,\Delta,\epsilon,S)$ is a Hopf-algebra. Let $\bar{S}=\f{M}(S):C\to C^\r{op}$, where $C^\r{op}$ is
identified with $\f{M}(B^\r{op},A)$ (see Lemma \ref{L2.1}). We shall show that $\bar{S}$ is an antipode for
the bialgebra $(C,\bar{\Delta},\bar{\epsilon})$. Suppose $b\in B$ and $f:A\to\m{K}$ is a $\m{K}$-module homomorphism. Let
$c=(f\ot\i_C)\phi(b)$, and note that, by Proposition \ref{P2.1} (a), such elements generate $C$. Then,
\begin{equation*}
\begin{split}
\mu_C(\i_C\ot\bar{S})\bar{\Delta}(c)&=\mu_C(\i_C\ot\bar{S})\bar{\Delta}(f\ot\i_C)\phi(b)\\
&=(f\ot\mu_C)(\i_A\ot\i_C\ot\bar{S})(\i_A\ot\bar{\Delta})\phi(b)\\
&=(f\ot\mu_C)(\i_A\ot\i_C\ot\bar{S})(\mu_A\ot\i_{C\ot C})(\i_A\ot F\ot\i_C)(\phi\ot\phi)\Delta(b)\\
&=(f\mu_A\ot\mu_C)(\i_A\ot\i_C\ot\bar{S})(\phi\ot\phi)\Delta(b)\\
&=(f\mu_A\ot\mu_C)(\i_A\ot F\ot\i_C)(\i_{A\ot C}\ot\i_A\ot\bar{S})(\phi\ot\phi)\Delta(b)\\
&=(f\mu_A\ot\mu_C)(\i_A\ot F\ot\i_C)(\phi\ot[(\i_A\ot\bar{S})\phi])\Delta(b)\\
&=(f\mu_A\ot\mu_C)(\i_A\ot F\ot\i_C)(\phi\ot\phi S)\Delta(b)\\
&=(f\mu_A\ot\mu_C)(\i_A\ot F\ot\i_C)(\phi\ot\phi)(\i_B\ot S)\Delta(b)\\
&=(f\ot\i_C)\mu_{A\ot C}(\phi\ot\phi)(\i_B\ot S)\Delta(b)\\
&=(f\ot\i_C)\phi\mu_B(\i_B\ot S)\Delta(b)\\
&=(f\ot\i_C)\phi(\epsilon(b)1_B)\\
&=\epsilon(b)f(1_A)1_C\\
&=\bar{\epsilon}(c)1_C.
\end{split}
\end{equation*}
Analogously, we have $\mu_C(\bar{S}\ot\i_C)\bar{\Delta}(c)=\bar{\epsilon}(c)1_C$. Thus $\bar{S}$ is an antipode for $C$.
We have proved the following theorem.
\begin{theorem}\label{T5.1}
Let $A$ be a FR commutative algebra and let $B$ be a (cocommutative) bialgebra. Then,
$\f{M}(B,A)$ has a canonical (cocommutative) bialgebra structure. Moreover, $\f{M}(B,A)$ is a Hopf-algebra if $B$ is.
\end{theorem}
Analogue of this result is satisfied for $\f{M}^\r{c}$:
\begin{theorem}\label{T5.2}
Let $A$ be a FR commutative algebra and let $B$ be a Hopf-algebra (resp. bialgebra).
Then $\f{M}^\r{c}(B,A)$ has a canonical Hopf-algebra (resp. bialgebra) structure.
\end{theorem}
\begin{proof}
For every commutative algebra $D$, the set $\b{Alg}(B,A\ot D)$
has a canonical group (resp. monoid) structure and thus the functor $\b{Alg}(B,A\ot ?)$ is group (resp. monoid) valued.
By the analog of Proposition \ref{P2.1} (d) for $\f{M}^\r{c}$, the algebra $\f{M}^\r{c}(B,A)$ represents $\b{Alg}(B,A\ot ?)$, and thus
has a canonical Hopf-algebra (resp. bialgebra) structure (see \cite[Chapter I]{Milne1}).

The theorem can of course be proved by another method:
By Theorem \ref{T5.1}, $\f{M}(B,A)$ is a Hopf-algebra (resp. bialgebra). On the other hand, the algebra $\f{M}^\r{c}(B,A)$ is the quotient
of $\f{M}(B,A)$ by the commutator ideal. Since the commutator ideal is also a Hopf-ideal (resp. biideal) it follows that $\f{M}^\r{c}(B,A)$
is a quotient Hopf-algebra (resp. bialgebra).

In the case that $B$ is commutative, there is still another proof:
$B$ is a group (resp. monoid) object in $\b{Alg}^\r{c}$. By Theorem \ref{T2.2}, $\f{M}^\r{c}(?,A)$ transforms coproducts to coproducts
and thus transforms group (resp. monoid) objects to group (resp. monoid) objects.
\end{proof}
\begin{remark}\label{R4.1}
\begin{enumerate}
\item[(a)] It follows from Remark \ref{R2.2} (b) that the free (resp. tensor) product of finitely many copies of a (resp. commutative)
Hopf-algebra (or bialgebra) has a canonical Hopf-algebra (or bialgebra) structure. For example, the usual bialgebra structures on
$\m{K}[x_1,\cdots,x_n]=\f{M}(\m{K}[x],\m{K}^n)$ and $\m{K}^\r{c}[x_1,\cdots,x_n]=\f{M}^\r{c}(\m{K}[x],\m{K}^n)$ (\cite[Section 3.2]{Sweedler1})
are induced by the usual bialgebra structure on $\m{K}[x]$.
\item[(b)] It has been shown by Wang \cite{Wang2} that free product of finitely many (C*-algebraic) compact quantum groups is a
compact quantum group.
\end{enumerate}
\end{remark}
\section{Quantum group of gauge transformations}
In this short section we apply the construction of the preceding section to obtain a notion for Hopf-algebra of gauge transformations.
We refer the reader to \cite{BrzezinskiMajid1} for basic ideas and notions on gauge theory on noncommutative spaces.
Let $B$ be a Hopf-algebra with comultiplication $\Delta$ and counit $\epsilon$ and let $E$ be a trivial left quantum
vector bundle with fiber $V$ and base $A$ in the sense of \cite{BrzezinskiMajid1}. This means that $A$ is an algebra and $V$ is a left
$B$-comodule algebra with a coaction $\rho$ and $E$ is (isomorphic to) the algebra $V\ot A$. We slightly change the definitions in
\cite{BrzezinskiMajid1} and define $\gamma$ to be a quantum gauge transformation of $E$ if $\gamma$ is a morphism from $B$ to $A$.
Then it is natural to define a gauge filed just to be a morphism from $V$ to $A$. Now, it is clear that if $A$ is commutative and FR
then the Hopf-algebra $C=\f{M}(B,A)$ and the algebra $H=\f{M}(V,A)$, respectively, play the role of \emph{quantum group of gauge transformations}
and \emph{quantum space of gauge fields} in our dual setting. In the classical case, there is a canonical action of gauge transformations
on gauge fields. This is also the case in our setting: Let $\bar{\Delta},\bar{\epsilon}$ and $\phi$ be as in the preceding section
and let $\varphi=\f{m}(V,A)$. Let $\bar{\rho}:H\to C\ot H$ be the unique morphism satisfying
$(\i_A\ot\bar{\rho})\varphi=(\mu_A\ot\i_{C\ot H})F(\phi\ot\varphi)\rho$. We show that $\bar{\rho}$ is a coaction:
\begin{equation*}
\begin{split}
(\i_A\ot[(\bar{\Delta}\ot\i_H)\bar{\rho}])\varphi&=(\i_A\ot\bar{\Delta}\ot\i_H)(\i_A\ot\bar{\rho})\varphi\\
&=(\i_A\ot\bar{\Delta}\ot\i_H)(\mu_A\ot\i_{C\ot H})F(\phi\ot\varphi)\rho\\
&=(\mu_A\ot\i_{C\ot C\ot H})F([(\i_A\ot\bar{\Delta})\phi]\ot\varphi)\rho\\
&=(\mu_A\ot\i_{C\ot C\ot H})F([(\mu_A\ot\i_{C\ot C})F(\phi\ot\phi)\Delta]\ot\varphi)\rho\\
&=(\mu_A\ot\i_{C\ot C\ot H})F(\phi\ot\phi\ot\varphi)(\Delta\ot\i_V)\rho,
\end{split}
\end{equation*}
\begin{equation*}
\begin{split}
(\i_A\ot[(\i_C\ot\bar{\rho})\bar{\rho}])\varphi&=(\i_A\ot\i_C\ot\bar{\rho})(\i_A\ot\bar{\rho})\varphi\\
&=(\i_{A\ot C}\ot\bar{\rho})(\mu_A\ot\i_{C\ot H})F(\phi\ot\varphi)\rho\\
&=(\mu_A\ot\i_{C\ot C\ot H})F(\i_{A\ot C}\ot\i_A\ot\bar{\rho})(\phi\ot\varphi)\rho\\
&=(\mu_A\ot\i_{C\ot C\ot H})F(\phi\ot[(\i_A\ot\bar{\rho})\varphi])\rho\\
&=(\mu_A\ot\i_{C\ot C\ot H})F(\phi\ot[(\mu_A\ot\i_{C\ot H})F(\phi\ot\varphi)\rho])\rho\\
&=(\mu_A\ot\i_{C\ot C\ot H})F(\phi\ot\phi\ot\varphi)(\i_B\ot\rho)\rho.
\end{split}
\end{equation*}
Thus $(\i_A\ot[(\bar{\Delta}\ot\i_H)\bar{\rho}])\varphi=(\i_A\ot[(\i_C\ot\bar{\rho})\bar{\rho}])\varphi$ and it follows that
$(\bar{\Delta}\ot\i_H)\bar{\rho}=(\i_C\ot\bar{\rho})\bar{\rho}$. It remains to prove counit identity $(\bar{\epsilon}\ot\i_H)\bar{\rho}=\i_H$.
It follows from the calculation below.
\begin{equation*}
\begin{split}
(\i_A\ot[(\bar{\epsilon}\ot\i_H)\bar{\rho}])\varphi&=(\i_A\ot\bar{\epsilon}\ot\i_H)(\i_A\ot\bar{\rho})\varphi\\
&=(\i_A\ot\bar{\epsilon}\ot\i_H)(\mu_A\ot\i_{C\ot H})F(\phi\ot\varphi)\rho\\
&=(\mu_A\ot\i_H)F([(\i_A\ot\bar{\epsilon})\phi]\ot\varphi)\rho\\
&=\varphi(\epsilon\ot\i_V)\rho\\
&=\varphi.
\end{split}
\end{equation*}
\section{Pontryagin dual of a FR Hopf-algebra}
In this section we consider, in our dual formalism, a construction analogous to the construction of (semi)group of
homomorphisms from a (semi)group to an abelian (semi)group, and as a consequence we obtain a notion of Pontryagin dual for FR
Hopf-algebras.

Let $(B,\Delta,\epsilon)$ be a cocommutative bialgebra and let $A$ be a FR commutative
algebra together with an arbitrary morphism $\Gamma:A\to A\ot A$.
Suppose that $B$ is free as module. Then, by Theorem \ref{T3.3}, the class $\c{M}_3$ has a universal family,
which we denote by $\psi:B\to A\ot D$. Let $\chi:B\to A\ot(D\ot D)$ be the composition
$(\mu_A\ot\i_{D\ot D})(\i_A\ot F\ot\i_{D})(\psi\ot\psi)\Delta$.
It follows from coassociativity and cocommutativity of $\Delta$ that,
\begin{equation*}
\begin{split}
&(\i_{A\ot A}\ot\mu_{D\ot D})(\i_A\ot F\ot\i_{D\ot D})(\chi\ot\chi)\Delta=\\
&(\mu_{A\ot A}\ot\i_{D\ot D})(\i_{A\ot A}\ot F\ot\i_{D})([(\Gamma\ot\i_{D})\psi]\ot[(\Gamma\ot\i_{D})\psi])\Delta=\\
&(\Gamma\ot\i_{D\ot D})\chi.
\end{split}
\end{equation*}
So, $\chi$ is in $\c{M}_3$. Since $\psi$ is the universal family of $\c{M}_3$, there is a unique morphism $\tilde{\Delta}:D\to D\ot D$ for which
$\chi=(\i_A\ot\tilde{\Delta})\psi$. In a manner analogous to the proofs of Lemma \ref{L5.1}
and Lemma \ref{L5.2}, it is shown that $\tilde{\Delta}$ is a cocommutative comultiplication. It is easily checked that
the family $\epsilon':B\to A\ot\m{K}$ defined by $b\mapsto\epsilon(b)1$ belongs to $\c{M}_3$. So, by the universal property of $D$,
there is a unique character $\tilde{\epsilon}$ on $D$ for which $(\i_A\ot\tilde{\epsilon})\psi=\epsilon'$. A proof analogous to the proof
of Lemma \ref{L5.3} shows that $\tilde{\epsilon}$ is a counit for $\tilde{\Delta}$. Moreover, if  the bialgebra $B$
has an antipode $S$, then appropriately modified versions of Lemma \ref{L2.1} and above proofs show that $D$ has an antipode.
So, we have proved theorem below.
\begin{theorem}\label{T5.3}
Let $B$ be a cocommutative bialgebra with  comultiplication $\Delta$, let $A$ be a FR commutative algebra, and let
$\Gamma:A\to A\ot A$ be a morphism. Suppose that $B$ is free as $\m{K}$-module and let $D$ be the parameter-algebra of the universal
family of the class of families $\psi':B\to A\ot D'$ for which
$(\Gamma\ot\i_{D'})\psi'=(\i_{A\ot A}\ot\mu_{D'})(\i_A\ot F\ot\i_{D'})(\psi'\ot\psi')\Delta$.
Then $D$ has a canonical cocommutative bialgebra structure. Moreover, $D$ is a Hopf-algebra if $B$ is.
\end{theorem}
A similar result is also satisfied for the universal family of $\c{M}^\r{c}_3$.

We want now offer a notion for \emph{Pontryagin dual} of a FR commutative Hopf-algebra.
For clarity of the discussion, we first recall below some standard facts.
Given a FR Hopf-algebra $A$ there corresponds another FR Hopf-algebra which we
call \emph{algebraic dual} of $A$ and denote by $A^*$.
The underlying module of $A^*$ is $\mathrm{Hom}_\m{K}(A,\m{K})$, the set of all module homomorphisms from $A$ to $\m{K}$. (Co)multiplication
of $A^*$ is induces by that of $A$ via canonical isomorphism
$\mathrm{Hom}_\m{K}(A\ot A,\m{K})\simeq\mathrm{Hom}_\m{K}(A,\m{K})\ot\mathrm{Hom}_\m{K}(A,\m{K})$. (Co)unit and antipode are induced
by the usual duality between $A$ and $\mathrm{Hom}_\m{K}(A,\m{K})$. Note that $A^{**}\simeq A$ as Hopf-algebras.
For a finite group $G$ the group algebra $\m{K}G$ (with convolution multiplication) is also a FR Hopf-algebra with
comultiplication, counit and antipode respectively given by $1_g\mapsto 1_g\ot 1_g$, $1_g\mapsto 1$ and $1_g\mapsto 1_{g^{-1}}$,
where $1_g:G\to\m{K}$ is defined by $1_g(g)=1$ and $1_g(g')=0$ for $g\neq g'$. Then the algebraic dual of $\m{K}G$ is canonically
isomorphic to the function algebra on $G$, denoted by $\m{K}(G)$, with pointwise multiplication and with comultiplication, counit and antipode
respectively given by $1_g\mapsto\sum_{h\in G}1_h\ot 1_{h^{-1}g}$, $1_g\mapsto 1_g(e)$ and $1_g\mapsto 1_{g^{-1}}$.

Recall that, in classical Harmonic Analysis, to any compact (resp. discrete) abelian group $G$ there corresponds its pontryagin dual
$\hat{G}$ which is a discrete (resp. compact) abelian group. $\hat{G}$ is the pointwise multiplication group of all group
homomorphisms from $G$ to the multiplicative group of complex numbers of absolute value 1. Similar to algebraic dual of Hopf-algebras
we have Pontryagin's duality: $G\simeq\hat{\hat{G}}$. It is not hard to see that if $G$ is finite
then $G\simeq\hat{G}$, $\m{C}G\simeq\m{C}(\hat{G})$ and $\m{C}(G)\simeq\m{C}\hat{G}$ (\cite{Majid1}).
Also note that in this case any group homomorphisms from $G$ to $\m{C}-\{0\}$ take values automatically on the unit circle.

A purely algebraic analog of the group $\m{C}-\{0\}$ is a Hopf-algebra $(\m{K}^\dag,\Delta,\epsilon,S)$ that represents the affine
group scheme of invertible elements of $\m{K}$-algebras (\cite{Milne1}). The underlying algebra of $\m{K}^\dag$ is $\m{K}^\r{c}[x,y]/(xy-1)$ and
the Hopf-algebra operations are defined by $\Delta(x)=x\ot x$, $\Delta(y)=y\ot y$, $\epsilon(x)=\epsilon(y)=1$, $S(x)=y$ and $S(y)=x$.
Note that $\m{K}^\dag$ is a free module.

Now we are in a position to introduce the notion of Pontryagin dual:
Let $A$ be a FR commutative Hopf-algebra. Denote by $\f{p}:\m{K}^\dag\to A\ot\f{P}(A)$ the universal family defined by
Theorem \ref{T5.3} with $B=\m{K}^\dag$. The corresponding universal family with commutative parameter-algebra is denoted by
$\f{p}^\r{c}:\m{K}^\dag\to A\ot\f{P}^\r{c}(A)$. Then, it is reasonable to call both Hopf-algebras $\f{P}(A)$ and $\f{P}^\r{c}(A)$
the \emph{Pontryagin dual} of $A$.

It is not clear to the author that in general how $A^*$ and $\f{P}(A)$ are related. But we have the following observation
in the special case of finite groups. Let $G$ be a finite abelian group and suppose that $\m{K}$ is a field (or more generally,
a ring for which any equation $z^n=1$ has only finitely many solutions where $n$ is any divisor of $|G|$). Then the multiplicative
group $\hat{G}_\m{K}$ of all group homomorphisms from $G$ to the group of units of $\m{K}$ is finite. Let $A=\m{K}(G)$ and let the family
$\psi:\m{K}^\dag\to A\ot\m{K}(\hat{G}_\m{K})$ be defined by $\psi(x)=\sum\alpha(g)1_g\ot 1_\alpha$
and $\psi(y)=\sum\alpha(g)^{-1}1_g\ot 1_\alpha$ where the sums are taken over all $g\in G$ and $\alpha\in\hat{G}_\m{K}$.
It is easily checked that $\psi$ satisfies condition of Theorem \ref{T5.3} and so there is a unique morphism
$\varphi:\f{P}^\r{c}(A)\to\m{K}(\hat{G}_\m{K})$ for which $\psi=(\i_{\m{K}(G)}\ot\varphi)\f{p}^\r{c}$. Now if $\m{K}$
is an algebraically closed field then a deduction analogous to Remark \ref{R2.2} (e) shows that $\f{P}^\r{c}(A)/J\simeq\m{K}(\hat{G}_\m{K})$
and so if $\m{K}=\m{C}$ then $\f{P}^\r{c}(A)/J\simeq A^*$ as Hopf-algebras.
\section{Hopf-algebra structure of the family of all automorphisms}
In this section we consider a canonical bialgebra  structure on the parameter-algebra of the family of all endomorphisms
of a FR algebra. This construction is analogous to the semigroup of self-maps on an ordinary space.

Let $A$ be a FR algebra and let $C=\f{M}(A,A)$ and $\phi=\f{m}(A,A)$.
Let $\Gamma:C\to C\ot C$ be the unique morphism satisfying $(\i_A\ot\Gamma)\phi=\phi\circ\phi$. Then,
\begin{equation*}
\begin{split}
(\i_A\ot[(\Gamma\ot\i_{C})\Gamma])\phi&=(\i_A\ot\Gamma\ot\i_{C})(\i_A\ot\Gamma)\phi\\
&=(\i_A\ot\Gamma\ot\i_{C})(\phi\circ\phi)\\
&=(\i_A\ot\Gamma\ot\i_{C})(\phi\ot\i_{C})\phi\\
&=([(\i_A\ot\Gamma)\phi]\ot\i_{C})\phi\\
&=([(\phi\ot\i_{C})\phi]\ot\i_{C})\phi\\
&=(\phi\ot\i_{C\ot C})(\phi\ot\i_{C})\phi,
\end{split}
\end{equation*}
and also,
\begin{equation*}
\begin{split}
(\i_A\ot[(\i_{C}\ot\Gamma)\Gamma])\phi&=(\i_A\ot\i_{C}\ot\Gamma)(\i_A\ot\Gamma)\phi\\
&=(\i_A\ot\i_{C}\ot\Gamma)(\phi\ot\i_{C})\phi\\
&=(\phi\ot\Gamma)\phi\\
&=(\phi\ot\i_{C\ot C})(\i_A\ot\Gamma)\phi\\
&=(\phi\ot\i_{C\ot C})(\phi\ot\i_{C})\phi.
\end{split}
\end{equation*}
So, $(\i_A\ot[(\Gamma\ot\i_{C})\Gamma])\phi=(\i_A\ot[(\i_{C}\ot\Gamma)\Gamma])\phi$. This identity together with the universal
property of $C$ imply $(\Gamma\ot\i_{C})\Gamma=(\i_{C}\ot\Gamma)\Gamma$. Thus, $\Gamma$ is a comultiplication.
(One may also apply the method of the proof of \cite[Theorem 4.1]{Soltan1} using Proposition \ref{P2.1} (a).)
Now, let $\epsilon:C\to\m{K}$ be the unique character satisfying $(\i_A\ot\epsilon)\phi=\i_A\ot1$. Then,
\begin{equation*}
\begin{split}
(\i_A\ot[(\i_{C}\ot\epsilon)\Gamma])\phi&=(\i_A\ot\i_{C}\ot\epsilon)(\i_A\ot\Gamma)\phi\\
&=(\i_A\ot\i_{C}\ot\epsilon)(\phi\ot\i_{C})\phi\\
&=(\phi\ot\i_\m{K})(\i_A\ot\epsilon)\phi\\
&=\phi\ot1\\
&=(\i_A\ot[\i_{C}\ot1])\phi.
\end{split}
\end{equation*}
This equality together with the universal property of $C$ imply $(\i_{C}\ot\epsilon)\Gamma=\i_{C}\ot1$.
Analogously, $(\epsilon\ot\i_{C})\Gamma=\i_{C}\ot1$. So, $\epsilon$ is a counit for $\Gamma$ and $(C,\Gamma,\epsilon)$
is a bialgebra. Also, it is easily seen that $A$ is a $C$-comodule via $\phi$.
We now show that $\phi$ is a universal coaction among all coactions of bialgebras on $A$.
Let $D$ be a bialgebra with a comultiplication $\Theta$, and suppose that we are given a coaction $\varphi:A\to A\ot D$ of $D$ on $A$.
Let $\psi:C\to D$ be the unique morphism satisfying $(\i_A\ot\psi)\phi=\varphi$. Then,
\begin{equation*}
\begin{split}
(\i_A\ot[(\psi\ot\psi)\Gamma])\phi&=(\i_A\ot\psi\ot\psi)(\i_A\ot\Gamma)\phi\\
&=(\i_A\ot\psi\ot\psi)(\phi\ot\i_{C})\phi\\
&=([(\i_A\ot\psi)\phi]\ot\psi)\phi\\
&=(\varphi\ot\psi)\phi\\
&=(\varphi\ot\i_D)(\i_A\ot\psi)\phi\\
&=(\varphi\ot\i_D)\varphi\\
&=(\i_A\ot\Theta)\varphi\\
&=(\i_A\ot\Theta)(\i_A\ot\psi)\phi\\
&=(\i_A\ot\Theta\psi)\phi
\end{split}
\end{equation*}
The above identity together with the universal property of $C$ imply $\Theta\psi=(\psi\ot\psi)\Gamma$, that means $\psi$ is
a bialgebra morphism. So, we have proved the theorem below.
\begin{theorem}\label{T6.1}
Let $A$ be a FR algebra. Then, $\f{M}(A,A)$ is a bialgebra in a canonical way and $A$ is an $\f{M}(A,A)$-comodule via $\f{m}(A,A)$.
Moreover, $\f{m}(A,A)$ is a universal coaction in the following sense:
For any bialgebra $D$ and any coaction $\varphi:A\to A\ot D$ of $D$ on $A$, the unique morphism $\psi:\f{M}(A,A)\to D$, satisfying
$(\i_A\ot\psi)\f{m}(A,A)=\varphi$, is also a bialgebra morphism and is compatible with the comodule structures on $A$.
\end{theorem}
Analogue of this result is also satisfied for $\f{M}^\r{c}(A,A)$.

Now, we consider the family of isomorphisms.
Let $G=\f{I}(A,A)$, and $\xi=\f{i}(A,A)$. Since $\xi\circ\xi:A\to A\ot(G\ot G)$ and $\i_A\ot1:A\to A\ot\m{K}$
are families of isomorphisms from $A$ onto $A$, there are unique morphisms $\Lambda:G\to G\ot G$ and $\delta:G\to\m{K}$
which satisfy $(\i_A\ot\Lambda)\xi=\xi\circ\xi$ and $(\i_A\ot\delta)\xi=\i_A\ot1$.
A proof analogous to the above proof shows that,
$$(\i_A\ot[(\i_{G}\ot\Lambda)\Lambda])\xi=(\xi\ot\i_{G\ot G})(\xi\ot\i_{G})\xi=(\i_A\ot[(\i_{G}\ot\Lambda)\Lambda])\xi.$$
Since the middle term of the above equality is an invertible family, the universal property of $G$ shows that
$(\Lambda\ot\i_{G})\Lambda=(\i_{G}\ot\Lambda)\Lambda$. Also, we may similarly conclude,
$(\delta\ot\i_{G})\Lambda=\i_{G}\ot1$. So, we have proved that $(G,\Lambda,\delta)$ is a bialgebra.
Analogously, if we use $G^\r{c}=\f{I}^\r{c}(A,A)$ and $\xi^\r{c}=\f{i}^\r{c}(A,A)$ instead of $G$ and $\xi$ above,
we find a bialgebra $(G^\r{c},\Lambda^\r{c},\delta^\r{c})$ which is also a quotient of $G$.
Actually, $G^\r{c}$ is a Hopf-algebra algebra: Let $S:G^\r{c}\to G^\r{c}$ be the unique morphism satisfying
$(\xi^\r{c})^{-1}=(\i_A\ot S)\xi^\r{c}$. Then,
\begin{equation*}
\begin{split}
(\i_A\ot[\mu_{G^\r{c}}(S\ot\i_{G^\r{c}})\Lambda^\r{c}])\xi^\r{c}&=
(\i_A\ot\mu_{G^\r{c}})(\i_A\ot S\ot\i_{G^\r{c}})(\i_A\ot\Lambda^\r{c})\xi^\r{c}\\
&=(\i_A\ot\mu_{G^\r{c}})(\i_A\ot S\ot\i_{G^\r{c}})(\xi^\r{c}\circ\xi^\r{c})\\
&=(\i_A\ot\mu_{G^\r{c}})(\i_A\ot S\ot\i_{G^\r{c}})(\xi^\r{c}\ot\i_{G^\r{c}})\xi^\r{c}\\
&=(\i_A\ot\mu_{G^\r{c}})((\xi^\r{c})^{-1}\ot\i_{G^\r{c}})\xi^\r{c}\\
&=\i_A\ot1\\
&=(\i_A\ot\delta^\r{c}1)\xi^\r{c}.
\end{split}
\end{equation*}
Since $\i_A\ot1:A\to A\ot G^\r{c}$ is an invertible family, the above identity together with the universal property of $G^\r{c}$
show that $\mu_{G^\r{c}}(S\ot\i_{G^\r{c}})\Lambda^\r{c}=\delta^\r{c}$. Similarly it is proved,
$\mu_{G^\r{c}}(\i_{G^\r{c}}\ot S)\Lambda^\r{c}=\delta^\r{c}$. So, $S$ is an antipode and $(G^\r{c},\Lambda^\r{c},\delta^\r{c},S)$
is a Hopf-algebra. Also, note that $A$ is a $G^\r{c}$-comodule via the coaction $\xi^\r{c}$. If $D$ is another commutative Hopf-algebra
with a coaction, $\varphi:A\to A\ot D$, on $A$, then Example \ref{E3.1} (b) shows that $\varphi$ is an invertible family and so
there is a canonical morphism from $G^\r{c}$ to $D$ which is, analogous to above, a Hopf-algebra morphism. So, we have proved
the next theorem.
\begin{theorem}\label{T8.2}
Let $A$ be a FR algebra. Then, $\f{I}^\r{c}(A,A)$ has a canonical Hopf-algebra structure and $A$ is a $\f{I}^\r{c}(A,A)$-comodule
via $\f{i}^\r{c}(A,A)$. Moreover, $\f{i}^\r{c}(A,A)$ is a universal coaction in the following sense:
If $D$ is a commutative Hopf-algebra with a coaction $\varphi:A\to A\ot D$ on $A$, then the unique morphism $\psi:\f{I}^\r{c}(A,A)\to D$,
satisfying $\varphi=(\i_A\ot\psi)\f{i}^\r{c}(A,A)$, is also a Hopf-algebra morphism and is compatible with the coactions on $A$.
\end{theorem}
At the end of this note, we consider a notion dual to the notion of Galois group. Let $A$ be a FR algebra and let $B$ be a subalgebra of $A$.
The results of Section 4 (see (1), Theorem \ref{T3.1} and Theorem \ref{T3.5} in the case $i=1$) shows that there is
commutative algebra $\f{G}$ together with a morphism $\f{g}:A\to A\ot \f{G}$, satisfying $\f{g}(b)=b\ot1$
for every $b\in B$, and with the following universal property: For every commutative algebra $D$ and any morphism $\varphi:A\to A\ot D$,
satisfying $\varphi(b)=b\ot 1$ for every $b\in B$, there is a unique morphism $\psi:\f{G}\to D$ such that $\varphi=(\i_A\ot\psi)\f{g}$.
A proof, analogous to the proof of Theorem \ref{T8.2}, shows that $\f{G}$ has a canonical Hopf-algebra structure and also $A$
is a $\f{G}$-comodule via $\f{g}$. Moreover, if in the above mentioned universal property, $D$ is a Hopf-algebra and $\varphi$ is a coaction
then $\psi$ is a Hopf-algebra morphism. A result analogous to the second part of Theorem \ref{T3.1} shows that there is a one-to-one
correspondence between characters on $\f{G}$ and algebra automorphisms of $A$ which preserve $B$. Indeed, the usual Galois group $\r{Gal}(A/B)$
of the algebra extension $A/B$ is isomorphic to the character group of the Hopf-algebra $\f{G}$.
Also, an expression, analogous to Remark \ref{R2.2} (e), shows that if $\m{K}$ is an algebraically closed field then $\f{G}/J$ is isomorphic
to $\m{K}(\r{Gal}(A/B))$ as Hopf-algebras, where $J$ denotes the Jacobson radical. Summarizing, it is reasonable to call $\f{G}=\f{Gal}(A/B)$
Galois-Hopf-algebra of the algebra extension $A/B$.

\bibliographystyle{amsplain}

\end{document}